\documentclass{amsart}
\usepackage[dvips]{color}
\usepackage{graphicx}
\usepackage{latexsym}
\usepackage{amssymb}
\usepackage{amsmath}
\usepackage{amsthm}
\usepackage{amscd}
\theoremstyle{plain}
\newtheorem{thm}{Theorem}
\newtheorem{lem}{Lemma}
\newtheorem{Bob}{Definition}

\theoremstyle{remark}
\newtheorem{rem}{Remark}

\begin{document}
\
\author[J.\ Chaika]{Jon Chaika}

\address{Department of Mathematics, Rice University, Houston, TX~77005, USA}

\email{jonchaika@math.uchicago.edu}

\title{There exists a topologically mixing Interval Exchange Transformation}

\maketitle
\begin{abstract}
We prove the existence of a topologically mixing interval exchange transformation and that no interval exchange is topologically mixing of all orders.
\end{abstract}

\section{Introduction}
The main result of this paper is that there exists a topologically mixing IET answering a question of Boshernitzan. The construction is based on the work of Keane \cite{nonue}. 
The plan of the paper is as follows: First, we state the results of the paper and introduce terminology. Second, we describe the Keane construction presented in \cite{nonue}, which is used to prove topological mixing. Third, we prove topological mixing. Last, we prove that no IET is topologically mixing of all orders.
\begin{Bob} Let $\Delta_{n-1}=\{(l_1,...,l_n):l_i>0, l_1+...+l_n=1 \} $ be the $(n-1)$-dimensional simplex. Given $L=(l_1,l_2,...,l_n)\in \Delta_{n-1}$
 we can obtain n subintervals of the unit
interval: $I_1=[0,l_1) ,
I_2=[l_1,l_1+l_2),...,I_n=[l_1+...+l_{n-1},1)$. If we are also given
 a permutation on n letters $\pi$ we obtain an n-\emph{Interval Exchange Transformation}  $ T \colon [0,1) \to
 [0,1)$ which exchanges the intervals $I_i$ according to $\pi$. That is, if $x \in I_j$ then \begin{center} $T(x)= x - \underset{k<j}{\sum} l_k +\underset{\pi(k')<\pi(j)}{\sum} l_{k'}$.\end{center}
\end{Bob}
Interval exchange transformation will henceforth be abbreviated IET.
\begin{rem} The various conventions for naming the permutation of IETs can be a source of confusion. The naming convention this paper uses says that if an IET has permutation $(4213)$ then the 4$^{th}$ interval is placed first by $T$, the 2$^{nd}$ interval is placed second the 1$^{st}$ is placed third and the 3$^{rd}$ is placed last. It does \emph{not} say that the 4$^{th}$ interval gets sent to the 2$^{nd}$.
\end{rem}

\begin{Bob} Let $X$ be a topological space. A dynamical system $T: X \to X$ is said to be \emph{topologically mixing} if for any two nonempty open sets $U$, $V$ there exists $N_{U,V}:=N$ such that $T^n(U) \cap V \neq \emptyset$ for all $n \geq N$.
\end{Bob}
\begin{thm} \label{main} There exists a topologically mixing 4-IET.
\end{thm}
\begin{rem} It has been shown that no IET satisfies the stronger property of being measure theoretically mixing \cite{nomixing}. It was shown \cite{fer} that minimal, uniquely ergodic dynamical systems of linear block growth can not be measure theoretically mixing. No 3-IET can be topologically mixing \cite{bc iet}. However, almost every IET that is not of rotation type satisfies the weaker property of being topologically weak mixing \cite{top w.mix}. In fact, almost every IET that is not of rotation type is measure theoretically weak mixing \cite{w.mix}. This was earlier shown for almost every 3-IET \cite{kat step}, almost every type W IET \cite{metric} and every type W IET satisfying an explicit Diophantine condition (Boshernitzan type, which almost every IET satisfies) \cite{bosh nog}. For type W IETs it has been shown that idoc (the Keane condition) implies topological weak mixing \cite{cdk}. The formulation of type W used in the two previous results is in \cite{CN}.
\end{rem}
\begin{rem} Since writing this paper we were able able to show that a residual (dense $G_{\delta}$) set of IETs with permutation $(n,n-1,...,1)$ are topologically mixing for any $n\geq 4$. The methods of this proof are different.
\end{rem}

\begin{rem} The IET exhibited will be topologically mixing with respect to the standard topology on [0,1) and also the finer topology given by considering the IET as a symbolic dynamical system. As a symbolic dynamical system the IET is continuous.
\end{rem}
A complementary result to Theorem \ref{main} is:
\begin{thm} No IET is topologically mixing of all orders.
\end{thm}
\begin{rem} Another result along these lines is proved in \cite{dk}, which shows that no substitution dynamical system can be measure theoretically mixing, but that there exists a topologically mixing substitution dynamical system. However, no substitution dynamical system can be topologically mixing of all orders. Some substitution dynamical systems have IETs that are obviously conjugate to them. However, the topologically mixing substitution dynamical system in \cite{dk} is a very particular substitution on two letters, which has a closely related substitution that is not topologically mixing. The IETs obviously corresponding to substitutions on two (or even three) letters are never topologically mixing.
\end{rem}

\begin{Bob} Given an IET $T: [0,1) \to [0,1)$ and a subinterval $J$,
let ${T|_J \colon J\to J}$ denote the first return map induced by $T$ on $J$. That is, if $x\in J$ let $$r(x)=\min\{n>0: T^n(x) \in J\} \text{ and }  T|_J(x)=T^{r(x)}(x) \text{ if }r(x) \text{ is finite.}$$  
 A tower over $J$ is the consecutive images of a subinterval before it returns to $J$. We call the individual images levels.
\end{Bob}
\begin{rem} The discontinuities for the induced map of an IET on an interval $[a,b)$ come from the preimages of discontinuities the IET and preimages of the endpoint. Therefore, if one chooses the endpoints carefully one gets an IET on at most the same number of intervals.
\end{rem}
\begin{Bob}\label{name} $I_j^{(1)}$ denotes the $j^{th}$ subinterval of the IET given by inducing on the $4^{th}$ interval. $I_j^{(k+1)}$ denotes the $j^{th}$ subinterval of the IET given by inducing on $I_4^{(k)}$.
\end{Bob}
\begin{Bob} Let $O(I_j^{(k)})$ denote the disjoint images under $T$ of $I_j^{(k)}$ before first return to $I^{(k)}$.
\end{Bob}
\begin{rem} This is the tower over $I^{(k)}$ given by images of $I_j^{(k)}$.
\end{rem}
\begin{Bob} Set $e_1= \left[
\begin{array}{c}
1\\ 0\\ 0\\ 0 
\end{array} \right]$, $e_2= \left[
\begin{array}{c}
0\\ 1\\ 0\\ 0 
\end{array} \right]$, $e_3= \left[
\begin{array}{c}
0\\ 0\\ 1\\ 0 
\end{array} \right]$, $e_4= \left[
\begin{array}{c}
0\\ 0\\ 0\\ 1 
\end{array} \right]$.
\end{Bob}
\begin{rem} Some of the above terminology is nonstandard, but convenient for our purposes.
\end{rem}
\section{The Keane construction}
This construction is given in \cite{nonue} and it is included here for completeness.
Consider IETs with permutation $(4213)$. Observe that the second interval gets shifted by $l_4-l_1$. If this difference is small relative to $l_2$ then much of $I_2$ gets sent to itself. At the same time, pieces of $I_3$ do not reach $I_2$ until they have first reached $I_4$. This is the heart of the Keane construction. The details of the Keane construction are centered  around iterating this procedure. 
  Keane showed that by choosing the lengths appropriately one could ensure that $T|_{I_4^{(1)}}$ is a (2413) IET. Name the intervals in reverse order and we once again get a (4213) IET. Moreover, Keane showed that for any choice $m,n\in \mathbb{N}$ one can find an IET whose landing pattern of $I^{(1)}_j$ is given by the following matrix. 
\begin{center} $A_{m,n}=
\left(
\begin{array}{cccc}
0 & 0 & 1 &1 \\
m-1 & m & 0 & 0 \\
n & n & n-1 & n \\
1 & 1 & 1 & 1 \end{array} \right)$; $\quad$ $m, n \in \mathbb{N} =\{1, 2, ...\} $.
\end{center} 
In order to see this, pick lengths for $I^{(1)}$ and write them as a column vector. Now assign lengths to the original IET by multiplying this column vector by $A_{m,n}$. The induced map will travel according to this matrix by construction. For instance, if one chooses lengths $[\frac 1 4 ,\frac 1 4 ,\frac 1 4 ,\frac 1 4]$ for $I_4^{(1)}, ...,I_1^{(1)}$ one gets lengths of 
$$[\frac{2}{2+2m-1+4n-1+4}, \frac{2m-1}{2m+4n+4}, \frac{4n-1}{2m+4n+4}, \frac {4}{2m+4n+4}]$$ for the original IET (after renormalizing).
For any finite collection of matrices one can just iterate this construction (assign lengths for $I^{(k)}$ by multiplying the lengths of $I^{(k+1)}$ by $A_{m_k,n_k}$, multiply the resulting column vector by $A_{m_{k-1},n_{k-1}},...$). Compactness (of $\Delta_3$, which can be thought of as the parameterizing space of (4213) IETs) ensures that we can pass to an infinite sequence of these matrices.

Since the intervals are named in reverse order, the discontinuity (under the induced map) between $I_2^{(1)}$ and $I_3^{(1)}$ is given by $T^{-1}(\delta_1)$ where $\delta_1$ denotes the discontinuity between $I_1$ and $I_2$. As the first row of the matrix suggests ${I_1=T(I_4^{(1)} \cup I_3^{(1)})}$. 
The discontinuity (under the induced map) between $I_1^{(1)}$ and $I_2^{(1)}$ is given by $T^{-m}(\delta_2)$ where $\delta_2$ denotes the discontinuity between $I_2$ and $I_3$. As the second row of the matrix suggests $$I_2=T(I_2^{(1)} \cup I_1^{(1)})\cup T^2(I_2^{(1)} \cup I_1^{(1)}) \cup ...\cup T^{m-1}(I_2^{(1)} \cup I_1^{(1)}) \cup T^m (I_2).$$ 
The discontinuity (under the induced map) between $I_3^{(1)}$ and $I_4^{(1)}$ is given by $T^{-n-1}(\delta_3)$ where $\delta_3$ denotes the discontinuity between $I_3$ and $I_4$.
 As the third row of the matrix suggests $$I_3= T^m(I_1^{(1)}) \cup T^{m+1}(I_2^{(1)}) \cup T^2(I_4^{(1)} \cup I_3^{(1)}) \cup T^{m+1}(I_1^{(1)}) \cup T^{m+2}(I_2^{(1)}) \cup T^3(I_4^{(1)} \cup I_3^{(1)}) \cup$$
 $$... \cup T^{m+n-1}(I_1^{(1)}) \cup T^{m+n}(I_2^{(1)}) \cup T^n(I_4^{(1)} \cup I_3^{(1)}) \cup T^{m+n}(I_1^{(1)}) \cup T^{m+n+1}(I_2^{(1)}) \cup T^{n+1}(I_4^{(1)}).$$  $I_4=I_4^{(1)} \cup I_3^{(1)} \cup I_2^{(1)} \cup I_1^{(1)}$. 
As the columns of the matrix suggest, this is also $$I_4= T^{n+1}(I_3^{(1)}) \cup T^{m+n+1}(I_2^{(1)}) \cup T^{m+n}(I_1^{(1)}) \cup T^{n+2}(I_4^{(1)}).$$
 To summarize, the composition of $I_j$ can be given by the $j^{th}$ row of the matrix. The travel before first return of $I_j^{(1)}$ can be given by the $j^{th}$ column. Additionally, because the intervals were named in reverse order, the permutation of the induced map is once again $(4213)$.

It is important for this construction that everything be iterated. The composition of $I_j^{(k)}$ in pieces of $I^{(k+r)}$ is given by $e_j^{\tau} A_{m_k,n_k}...A_{m_{k+r-1},n_{k+r-1}}$ (where $e_j^{\tau}$ denotes the transpose of $e_j$). Likewise, the travel of $I_j^{(k+r)}$ under $T|_{I^{(k)}}$ before first return to $I^{(k+r)}$ is given by $A_{m_k,n_k}...A_{m_{k+r-1},n_{k+r-1}}e_j$.

Now for some explicit statements about the travel of subintervals of $I^{(k)}$ under the induced map $T|_{I^{(k)}}$. When $I_3^{(k)}$ returns to $I^{(k)}$ it entirely covers $I_4^{(k)}$. It is a subset of $I_3^{(k)} \cup I_4^{(k)}$.
When $I_4^{(k)}$ returns to $I^{(k)}$ it entirely covers $I_1^{(k)}$. It intersects $I_2^{(k)}$. Moreover part of this intersection will stay in $O(I_2^{(k)})$ for the next $m_{k}b_{k,2}$ images (the other part $(m_{k}-1)b_{k,2}$ images).
When $I_2^{(k)}$ returns to $I^{(k)}$ it intersects $I_3^{(k)}$. Moreover this piece of intersection will stay in $O(I_3^{(k)})$ for the next $n_{k}b_{k,3}$ images.

\begin{Bob} Let $b_{k,i}$ be the first return time of $I_i^{(k)}$ to $I^{(k)}$. 
\end{Bob}
\begin{rem} $b_{k,2}=m_{k-1}b_{k-1,2}+n_{k-1}b_{k-1,3}+b_{k-1,4}$ and $b_{k,3}=b_{k-1,1}+(n_{k-1}-1)b_{k-1,3}+b_{k-1,4}$.
\end{rem}

Some facts to keep in mind:
\begin{enumerate}
\item The choice of $n_k$ has no effect on $b_{i,2}$ for $i \leq k$.
\item The choice of $n_k$ has no effect on $b_{i,3}$ for $i \leq k$.
\item The choice of $m_k$ has no effect on $b_{i,2}$ for $i \leq k$.
\item The choice of $m_k$ has no effect on $b_{i,3}$ for $i \leq k+1$.
\end{enumerate}
\section{There exists a topologically mixing Keane IET}

Conditions for $b_{k,2}$ and $b_{k,3}$ to ensure topological mixing:
\begin{enumerate}
\item $b_{k,2}$ is prime for all $k$.
\item $b_{i,2} \not | b_{k,3}$ for all $i<k$.
\item The group of multiplicative units mod $\underset{i=1}{\overset{k}{\Pi}} b_{i,2}$ has more than $\frac 1 2 \underset{i=1}{\overset{k}{\Pi}} b_{i,2}$ elements.
\item $b_{k,2}b_{k+1,3}+b_{k+1,3}+b_{k,4}+b_{k-1,4}<m_{k}b_{k,2}.$ 
\item $b_{k,3}b_{k,2}+b_{k,2}<n_k b_{k,3}.$ 

\end{enumerate}
Theorem 1 will be proven by showing that any Keane IET  chosen in this way is topologically mixing. We first show that the set of such IETs is nonempty.
\begin{lem} We can choose $b_{k,2}$ and $b_{k,3}$ to fulfill these conditions.
\end{lem}
\begin{proof}
By induction. Assume we have chosen $n_1,m_1,n_2,m_2,...,n_{k-1},m_{k-1}$; we describe how to choose $n_k$ and then given this $n_k$ how to choose $m_k$. Consider congruence modulo $g:=\underset{i=1}{\overset{k}{\Pi}}b_{i,2}$. Choose a congruence class $[f]$ that is in the group of multiplicative units and so that $[f+b_{k,3}-b_{k,1}]$ is in the multiplicative group of units. This can be done by pigeon hole principle (by condition 3). Pick $n_k$ so that $b_{k+1,3} \in [f]$ and so that $n_k>\frac{b_{k,3}b_{k,2}+b_{k,2}}{b_{k,3}}$. This can be done because $b_{k,3}$ is relatively prime to the $b_{i,2}$. Next we pick $m_k$ so that $b_{k+1,2}$ is prime, $m_k>\frac{b_{k,2}b_{k+1,3}+b_{k,4}+b_{k+1,3}}{b_{k,2}}$ and condition 3 is satisfied. This can be done because we wish to find a prime in the arithmetic progression $n_{k}b_{k,3}+b_{k,4}+b_{k,2}\mathbb{N}$ and the starting point and the increment are relatively prime and the other conditions merely require choosing $m_k$ large enough. This is Dirichlet's Theorem (see for example \cite[Chapter 7]{dirichlet}).
\end{proof}
Let $c_k=b_{k+1,2}b_{k+2,3}+b_{k+2,3}+b_{k+2,4}+b_{k+1,4}$ and $d_k=b_{k+2,3}b_{k+2,2}+b_{k+2,2}$.

 Let $J$ be a subinterval in $I$ containing at least one level of a tower over $I^{(k)}$. This means that it contains at least 1 level from each of the 4 towers over $I^{(k+2)}$. For all $j>k$, $i>c_j$, $T^i(J)$ intersects every level of every tower over $I^{(j-1)}$. This is proved in the following lemmas. In these arguments it will be important to pick out a level from $O(I_2^{(k+2)})$ and $O(I_3^{(k+2)})$. These will be denoted $J'$.

\begin{lem} At times $c_k$ to $d_k$ $J'$, a level in $O(I_3^{(k+2)})$, intersects every level of $O(I_2^{(k+1)})$.
\end{lem}
\begin{proof} There exists $0<i \leq b_{k+2,3} $ (it is equal to $b_{k+2,3}$ for $I_3^{(k+2)}$ but for pieces of the orbit it is less than) such that $I_4^{(k+2)} \subset T^i(I_3^{(k+2)})$. So ${T^{i+b_{k+2,4}}(I_3^{(k+2)}) \cap I_2^{(k+2)} \neq \emptyset}$. Also $$T^{i+b_{k+2,4}+b_{k+1,4}}(I_3^{(k+2)}) \cap I_2^{(k+1)} \neq \emptyset.$$ In fact, $$T^{i+b_{k+2,4}+b_{k+1,4}+jb_{k+2,3}}(I_3^{(k+2)}) \cap I_2^{(k+1)} \neq \emptyset$$ for $j<n_{k+3}$.
 Pieces of $I_3^{(k+2)}$ are inserted into $I_2^{(k+1)}$ with a delay of $b_{k+2,3}$ which is coprime to $b_{k+1,2}$.  It follows that $T^{c_k}(J')$ intersects every level of $O(I_2^{(k+1)})$. By condition 5 it follows that $T^r(J')$ intersects every level of $O(I_2^{(k+1)})$ for $c_k \leq r \leq d_k$. Moreover, the pieces inserted take $m_{k+1}b_{k+1,2}$ to leave $O(I_2^{(k+2)})$. Because $m_{k+1}b_{k+1,2}>b_{k+2,3}b_{k+1,2}$ (condition 4) the piece does not leave $O(I_2^{(k+2)})$ before another is inserted into its level.
\end{proof}


\begin{lem} At times $d_k$ to $c_{k+1}$ $J'$, a level in $O(I_2^{(k+2)})$, intersects every piece of $O(I_3^{(k+2)})$.
\end{lem}
\begin{proof} There exists $0<i \leq b_{k+2,2} $ (it is equal to $b_{k+2,2}$ for $I_2^{(k+2)}$ but for pieces of the orbit it is less than) such that $I_3^{(k+2)} \cap T^i(I_2^{(k+2)}) \neq \emptyset$. Also $I_3^{(k+2)} \cap T^{i+jb_{k+2,2}} \neq \emptyset$ for $j<m_{k+2}$. Because $b_{k+2,2}$ is relatively prime to $b_{k+2,3}$ we have $T^{i+jb_{k+2,2}}(I_2^{(k+2)})$ intersects each level of $O(I_3^{(k+2)})$ for $ j = b_{k+2,3}$. It follows from condition 4 that $T^{r}(I_2^{(k+2)})$ intersects each level of $O(I_3^{(k+2)})$ for $d_k \leq r\leq c_{k+1}$. Moreover, the pieces inserted take $(n_{k+2}-1)b_{k+2,3}$ to leave $O(I_3^{(k+2)})$. Because $(n_{k+2}-1)b_{k+2,3}>b_{k+2,2}b_{k+2,3}$ (condition 5) the piece does not leave $O(I_3^{(k+1)})$ before another is inserted into its level.
\end{proof}
\begin{proof}[Proof of Theorem 1] Let $J_1, J_2$ be any two nonempty intervals in $I$. Therefore there exists $k_0$ such that both contain some level of a tower over $I^{(k_0)}$. This implies that they contain a level from each tower over $I^{(k)}$ for all $k>k_0+1$. This implies that $T^n(J_1) \cap J_2 \neq \emptyset$ for $n \in [c_{k},d_{k}]$ because $J_1$ contains a level of $I_3^{(k+2)}$ and $J_2$ contains a level of $I_2^{(k+1)}$. Also $T^n(J_1) \cap J_2 \neq \emptyset$ for $n \in [d_{k},c_{k+1}]$ because $J_1$ contains a level of $I_2^{(k+2)}$ and $J_2$ contains a level of $I_3^{(k+2)}$. It follows that $T^n(J_1) \cap J_2 \neq \emptyset$ for any $n>c_{k_0+1}$.
\end{proof}


\section{No IET is topologically mixing of all orders}
The argument is a straightforward application of \cite{nomixing}. Let $T$ be a $d$-IET.
Observe that a topologically mixing IET must be minimal (otherwise it splits into disjoint invariant components).
Let $J$, $J'$ be any disjoint intervals bounded by discontinuities of $T^l$ for some $l$,  and $n_1,...,n_{d^2}$ be natural numbers. We will find a violation of topological mixing of order $d^2+1$ at bigger times. Pick an interval $V$ such that all of the first returns to $V$ are greater than $\max\{l,n_1,...,n_{d^2}\}$. We may also choose $V$ so that $T|_V$ is an $s$-IET for some $s\leq d$. By our assumption that the return times to $V$ are larger than $l$, each level of a tower over $V$ is either contained in $J$ or disjoint from $J$. Let $U_1,U_2,...,U_s$ be its subintervals. $T|_{U_i}$ is an $s_i$-IET for $s_i \leq s$. Call its intervals $U_{i,1},...,U_{i,s_i}$ and their return times $r_{i,1},...,r_{i,s_i}$. If $x \in O(U_i)\cap J$ and $x \in O(U_{i,j})$ then $T^{r_{i,j}}(x) \in J$. This is because $x \in T^k(U_i) \subset J$ for some $k<r_i$, in fact $x \in T^{k}(U_{i,j})$. $T^{r_{i,j}-k}(x) \in U_i$. So $T^k(T^{r_{i,j}-k}(x)) \in T^k(U_i) \subset J$. Therefore $\underset{i,j=1 }{\overset{d}{\cap}}T^{r_{i,j}}(J) \cap J'=\emptyset$.


\section{Acknowledgments}

I would like to thank my advisor, M. Boshernitzan, for posing this problem, helpful conversations and encouragement. 
 I would also like to thank A. Bufetov and W. Veech for helpful conversations and encouragement. I would like to thank the referee for suggestions that improved the paper. I was supported by Rice University's Vigre grant DMS-0739420 and a Tracy Thomas award while working on this paper.

\end{document}